\documentclass[a4paper,12pt]{amsart}
\usepackage[width=0.8\paperwidth,height=0.85\paperheight,twoside=false,includehead,includefoot]{geometry}
\newtheorem{thm}{Theorem}[section]
\newtheorem{lem}[thm]{Lemma}

\newtheorem{prop}[thm]{Proposition}

\newtheorem{rem}[thm]{Remark}
\theoremstyle{definition}
\newtheorem{defn}[thm]{Definition}
\newcommand{\A}{\mathcal{A}}
\newcommand{\F}{\mathcal{F}}
\newcommand{\Q}{\mathbb{Q}}
\newcommand{\Z}{\mathbb{Z}}
\newcommand{\R}{\mathbb{R}}
\def\k{\textrm{\textup{\textmd{\textbf{k}}}}}

\font\sevency=wncyr7 \def\sh{\,\hbox{\sevency X}}

\font\sevency=wncyr7 \def\sh{\,\hbox{\sevency X}}

\def\ZF{\mathit{Z}_{\mathcal{F}}}
\begin{document}
\title[A note on derivation relations for MZVs and FMZVs]{A note on derivation relations for multiple zeta values and finite multiple zeta values}

\author{Yasunobu Horikawa}
\address[Yasunobu Horikawa]{Kurume University Junior / Senior High School,
 20-2, Nonakamachi, Kurume-shi, Fukuoka, 839-0862, Japan}
\email{horikawa\_\,yasunobu@kurume-u.ac.jp}

\author{Hideki Murahara}
\address[Hideki Murahara]{Nakamura Gakuen University,
 5-7-1, Befu, Jonan-ku Fukuoka-shi, Fukuoka, 814-0198, Japan} 
\email{hmurahara@nakamura-u.ac.jp}

\author{Kojiro Oyama}
\address[Kojiro Oyama]{1-31-17, Chuo, Aomori-shi, Aomori, 030-0822, Japan}
\email{k-oyama@kyudai.jp}

\begin{abstract}
 The derivation relation is a well known relation among multiple zeta values, which was first obtained by Ihara, Kaneko and Zagier. 
 The analogous formula for finite multiple zeta values, which we call the derivation relation for finite multiple zeta values, was conjectured by the third author and proved by the second author. 

 In this paper, we show these two kinds of derivation relations are respectively equivalent to the Ohno type relations for multiple zeta values and finite multiple zeta values. 
 We also reprove these derivation relations in several different ways. 
\end{abstract}
\maketitle

\section{Introduction}
\subsection{Multiple zeta(-star) values and Finite multiple zeta(-star) values}
 For $k_1,\dots,k_r\in \Z_{\ge1}$ with $k_1 \ge 2$, the multiple zeta values (MZVs) and the multiple zeta-star values (MZSVs) are defined by 
\begin{align*}
 \zeta(k_1,\dots, k_r)=\sum_{n_1>\cdots >n_r \ge 1} \frac {1}{n_1^{k_1}\cdots n_r^{k_r}}, \\
 \zeta^\star (k_1,\dots, k_r)=\sum_{n_1\ge \cdots \ge n_r \ge 1} \frac {1}{n_1^{k_1}\cdots n_r^{k_r}}. 
\end{align*} 
They are generalizations of the values of the Riemann zeta function at positive integers.

Among a large number of variants of the MZ(S)Vs, we consider two types of finite multiple zeta(-star) values (FMZ(S)Vs);
 $\A$-finite multiple zeta(-star) values ($\A$-FMZ(S)Vs) and
 symmetric multiple zeta(-star) values (SMZ(S)Vs). 
Set $\A=(\prod_{p}\Z/p\Z)/(\bigoplus_{p}\Z/p\Z)$, where $p$ runs over all primes.  
For $k_1, \dots ,k_r \in \Z_{\ge1}$, the $\A$-FMZVs and the $\A$-FMZSVs are defined by
\begin{align*}
 \zeta_{\mathcal{A}} (k_1, \dots, k_r)
 &=\biggl( \sum_{p>n_1>\cdots >n_r\ge1}
 \frac{1}{n_1^{k_1}\cdots n_r^{k_r}} \bmod{p} \biggr)_{p} \in\A, \\ 
 \zeta_{\mathcal{A}}^{\star} (k_1, \dots, k_r)&=\biggl( \sum_{p>n_1\ge\cdots\ge n_r\ge1}
 \frac{1}{n_1^{k_1}\cdots n_r^{k_r}} \bmod{p} \biggr)_{p} \in\A. 
\end{align*}
The SMZ(S)Vs was introduced by Kaneko and Zagier \cite{kaneko_2014, kaneko_zagier_2017}. 
For $k_1,\dots ,k_r \in \Z_{\ge1}$, we let  
\begin{align*}
 \zeta_{\mathcal{S}}^{\ast} (k_1, \dots, k_r)
 &=\sum_{i=0}^{r}(-1)^{k_1+\cdots +k_i}\zeta^{\ast}(k_i, \dots , k_1) \zeta^{\ast}(k_{i+1}, \dots,k_r). 
\end{align*}
Here, the symbols $\zeta^{\ast}$ on the right-hand sides stand for the regularized values coming from harmonic regularizations,
i.e.,  real values obtained by taking constant terms of harmonic regularizations as explained in Ihara-Kaneko-Zagier \cite{ihara_kaneko_zagier_2006}. 
In the sum, we understand $\zeta^{\ast}(\emptyset )=1$.
Let $\mathcal {Z}_{\R}$ be the $\Q$-vector subspace of $\R$ spanned by $1$ and all MZVs, which is a $\Q$-algebra. 
Then, the SMZVs is defined as an element 
in the quotient ring $\mathcal {Z}_{\R}/(\zeta (2))$ by 
\[
 \zeta_{\mathcal{S}} (k_1, \dots,k_r)=\zeta_{\mathcal{S}}^{\ast} (k_1, \dots,k_r) \bmod \zeta(2) . 
\]
For $k_1, \dots,k_r \in \Z_{\ge1}$, we also define the SMZSVs by  
\[ 
 \zeta_{\mathcal{S}}^{\star} (k_1, \dots,k_r)
 =\sum_{ \substack{\circ \textrm{ is either a comma "," } \\ \textrm{ or a plus "+"}} } 
 \zeta_{\mathcal{S}}^{\ast}(k_1 \circ \cdots \circ k_r) \bmod  \zeta(2)  \in\mathcal {Z}_{\R}/(\zeta (2)). 
\]   

Denoting $\mathcal {Z}_{\A}$ by the $\Q$-vector subspace of $\A$ spanned by $1$ and all $\A$-FMZVs, Kaneko and Zagier conjecture that $\mathcal {Z}_{\A}$ and $\mathcal {Z}_{\R}/(\zeta(2))$ are isomorphic as $\Q$-algebras via the correspondence $\zeta_{\mathcal{A}} (k_1, \dots, k_r) \leftrightarrow \zeta_{\mathcal{S}} (k_1, \dots,k_r)$. 
(For more details, see Kaneko-Zagier \cite{kaneko_2014, kaneko_zagier_2017}.) 
In what follows, we call $\A$-FMZVs and SMZVs as FMZVs.

\subsection{Main results}
The topics that we present in this paper are derivation relations and Ohno type relations. 
There are 5 types of relations we mainly discuss in this paper. 
\begin{itemize}
 \item Derivation relation for MZVs (Ihara-Kaneko-Zagier \cite{ihara_kaneko_zagier_2006}; Theorem \ref{der_MZVs})
 \item Ohno's relation for MZVs (Ohno \cite{ohno_99}; Theorem \ref{ohno})
 \item Ohno type relation for MZVs (Theorem \ref{ohno-type_MZVs2})
 \item Derivation relation for FMZVs (Murahara \cite{murahara_2016}; Theorem \ref{der_FMZVs})
 \item Ohno type relation for FMZVs (Oyama \cite{oyama_2015}; Theorem \ref{ohno-type_FMZVs})
\end{itemize}
The derivation relation proved in \cite{ihara_kaneko_zagier_2006} (Theorem \ref{der_MZVs}) is the well known relation among MZVs. 
As an analogue, the second auther \cite{murahara_2016} proved its counterpart for FMZVs (Theorem \ref{der_FMZVs}). 
On the other hand, Ohno \cite{ohno_99} obtained one of the most general explicit relations among MZVs. 
This is commonly reffered to as `Ohno's relation' (Theorem \ref{ohno}). 
The third auther \cite{oyama_2015} obtained its analogue for FMZVs (Theorem \ref{ohno-type_FMZVs}).  
In section $2$, we present its counterpart for MZVs (`Ohno type relation', Theorem \ref{ohno-type_MZVs2}).

In this paper, we will show the followings: 

\noindent $1.$ 
The equivalence between the derivation relation  
and the Ohno type relation for MZVs

\noindent $2.$ 
The equivalence between the derivation relation 
and the Ohno type relation for FMZVs

\noindent $3.$ Two alternative proofs of the derivation relation for MZVs

\noindent $4.$ Three alternative proofs of the derivation relation for FMZVs

 \begin{table}[h]
  \centering
  \begin{tabular}{|c|ccc|}
    \hline
     & Derivation rel. && Ohno type rel. \\
    \hline 
     MZVs & Theorem \ref{der_MZVs} (\cite{ihara_kaneko_zagier_2006}) &$\Leftrightarrow$& Theorem \ref{ohno-type_MZVs2} (this paper) \\
    \hline 
    FMZVs & Theorem \ref{der_FMZVs} (\cite{murahara_2016}) &$\Leftrightarrow$& Theorem \ref{ohno-type_FMZVs} (\cite{oyama_2015}) \\
    \hline
  \end{tabular}
 \end{table}

\section{Derivation relation and Ohno type relation for multiple zeta values and their equivalence}
\subsection{Derivation relation (MZVs)}
Let $\frak{H} =\Q \left\langle x,y \right\rangle$ be the noncommutative polynomial ring in two indeterminates $x$, $y$, and $\frak{H}^1$ (resp. $\frak{H}^0$) its subring $\Q +\frak{H} y$ (resp. $\Q +x \frak{H} y$). 
Set $z_{k} = x^{k-1} y$ $(k\in\Z_{\ge1})$. 
We define the $\Q$-linear map $\mathit{Z}\colon\frak{H}^0 \to \R$ by $\mathit{Z}(1)=1$ and $\mathit{Z}( z_{k_1} \cdots z_{k_r})= \zeta (k_1,\ldots, k_r)$. 

A derivation $\partial$ on $\frak{H}$ is a $\Q$-linear map $\partial\colon\frak{H}\to\frak{H}$ satisfying Leibniz's rule $\partial(ww^{\prime})=\partial(w)w^{\prime}+w\partial(w^{\prime})$. 
Such a derivation is uniquely determined by its images of generators $x$ and $y$. Set $z=x+y$. 
For each $l\in\Z_{\ge1}$, the derivation $\partial_l$ on $\frak{H}$ is defined by $\partial_l(x)=xz^{l-1}y$ and $\partial_l(y)=-xz^{l-1}y$. 
We note that $\partial_l(1)=0$ and $\partial_l(z)=0$. 
\begin{thm}[Derivation relation for MZVs; Ihara-Kaneko-Zagier \cite{ihara_kaneko_zagier_2006}] \label{der_MZVs}
For $l\in\Z_{\ge1}$, we have 
 \begin{align*} 
 \mathit{Z} (\partial_{l}(w)) = 0 \quad (w \in \frak{H}^0).  
 \end{align*}
\end{thm}

\subsection{Ohno's relation and Ohno type relation (MZVs)}
As described above, Ohno's relation is well known relation among MZVs. 
To state Ohno's relation, we need to define the dual index. 
\begin{defn}[Dual index] \label{du}  
Let $\k=(k_1,\ldots ,k_r) \in \Z_{\ge1}^r$ be an index with $k_1 \ge 2$. 
We write 
\[ \k=(a_1+1, \underbrace{1,\ldots ,1}_{b_1-1},\ldots ,a_s+1, \underbrace{1,\ldots ,1}_{b_s-1}) \] 
with $a_p, b_q\ge 1$. 
Then, the dual index of $\k$ is denoted by 
\[ (b_s+1, \underbrace{1,\ldots ,1}_{a_s-1},\ldots ,b_1+1, \underbrace{1,\ldots ,1}_{a_1-1}). \]
\end{defn}
\begin{thm}[Ohno's relation; Ohno \cite{ohno_99}] \label{ohno} 
For $\k =(k_1,\ldots,k_r) \in \Z_{\ge1}^r$ with $k_1 \ge2$ and $m\in\Z_{\ge 0}$, we have
\begin{align*}
 \sum_{\substack{e_1 + \cdots + e_r =m \\ e_i \ge 0 \, (1\le i \le r)}}
 \zeta (k_1+e_1,\ldots,k_r+e_r) 
 =\sum_{\substack{e_1' + \cdots + e_{r'}' = m \\ e_i' \ge 0  \, (1\le i \le r') }}
 \zeta (k_1'+e_1',\cdots,k_{r'}'+e_{r'}'), 
\end{align*}
where $(k'_1,\ldots ,k'_{r'})$ is the dual index of $\k$. 
\end{thm}
We introduce the results similar to Theorem \ref{ohno}.
The main difference between Theorem \ref{ohno} and the following Thoerem \ref{ohno-type_MZVs2} is whether to describe the statement by the dual index or the Hoffman's dual index.
Thoerem \ref{ohno-type_MZVs2} is essentially contained in Theorem \ref{ohno} and is equivalent to Theorem \ref{der_MZVs}.
\begin{defn}[Hoffman's dual index]
 For $\k =(k_1,\dots,k_r) \in \Z_{\ge1}^r$, we define Hoffman's dual index of $\k$ by
 \begin{align*}
  \k^{\vee}=(\underbrace{1,\dots,1}_{k_1}+\underbrace{1,\dots,1}_{k_2}+1,\dots,1+\underbrace{1,\dots,1}_{k_r}).
 \end{align*}
\end{defn}
\begin{thm}[Ohno type relation for MZVs] \label{ohno-type_MZVs2} 
 For $(k_1,\dots,k_r) \in \Z_{\ge1}^r$ and $m \in \Z_{\ge 0}$, we have
 \begin{align*}
  \sum_{\substack{e_1 + \cdots + e_r =m \\ e_i \ge 0 \, (1\leq i \le r)}} 
  \zeta (k_1+e_1+1,k_2+e_2,\dots,k_r+e_r)
  =\sum_{\substack{e_1' + \cdots + e_{r'}' = m \\ e_i' \ge 0 \, (1\leq i \le r') }}
  \zeta ((1,k'_1+e'_1,\dots,k'_{r'}+e'_{r'})^{\vee}),  
 \end{align*}
where $(k'_1,\ldots ,k'_{r'})$ is Hoffman's dual index of $(k_1,\dots,k_r)$. 
\end{thm}

\subsection{The equivalence of derivation relation and Ohno type relation (MZVs)}
In this subsection, we prove the following:
\begin{thm} \label{equivalence1}
 The derivation relation (Theorem \ref{der_MZVs}) and the Ohno type relation (Theorem \ref{ohno-type_MZVs2})
 are equivalent. 
\end{thm}

We define $\Q$-linear map $R\colon\frak{H}y \to \frak{H}y$ by $R(z_{k_1}\dots z_{k_r})=z_{k_r}\dots z_{k_1}$. 
Let $\tau$ be the anti-automorphism of $\frak{H}$ that interchanges $x$ and $y$. 
We also let $\alpha$ be the automorphism of $\frak{H}$ that interchanges $x$ and $y$, and $\tilde{\alpha}$ be the $\Q$-linear map satisfying $\tilde{\alpha} (wy) = \alpha (w)y$.
We define the map $\sigma$ as an automorphism of $\widehat{\frak{H}}=\Q\left\langle\left\langle x,y\right\rangle\right\rangle$ satisfying $\sigma(x)=x$ and $\sigma(y)=\frac{1}{1-x}y$, and the $\Q$-linear map $\sigma_m\colon\frak{H} \to \frak{H}$ as the homogeneous degree $m$ components of $\sigma$. 
(See also \cite[Section $6$]{ihara_kaneko_zagier_2006} and \cite[Appendix]{tanaka_2009}.)
We also define the $\Q$-linear operator $L_w$ on $\frak{H}$ by $L_w(1)=1$ and $L_w(w')=ww' \,\, (w,w' \in\frak{H}, w'\ne 1)$.

We can easily confirm the following proposition. 
\begin{prop}
 For $w\in x\frak{H}y$, we have
 \[
  \tau (w)=L_x R \tilde{\alpha} L_x^{-1} (w).
 \]
\end{prop}

\begin{proof}[Proof of Theorem \ref{equivalence1}] 
 According to Ihara-Kaneko-Zagier \cite{ihara_kaneko_zagier_2006}, the derivation relation is denoted by 
 $\mathit{Z} ((\sigma_m-\tau\sigma_m\tau)L_x(w))=0$ for $w\in \frak{H}y$. 
 Similarly, we find the Ohno type relation can be stated by 
 $\mathit{Z} (L_x(\sigma_m-\tilde{\alpha}\sigma_m\tilde{\alpha})(w))=0$ for $w\in\frak{H}y$. 
 We note that $L_x \sigma=\sigma L_x$, $R\sigma_m=\sigma_m R$, and $R\tilde{\alpha}=\tilde{\alpha}R$.
 The following equivalence proves the theorem: 
 \begin{align*}
 &\mathit{Z} ( (\sigma_m-\tau \sigma_m \tau) L_x  (w) )=0 \\
 \iff &\mathit{Z} ( L_x(\sigma_m-L_x^{-1}\tau \sigma_m \tau L_x) (w) )=0 \\
 \iff &\mathit{Z} ( L_x(\sigma_m-L_x^{-1}(L_x R \tilde{\alpha} L_x^{-1})\sigma_m (L_x R \tilde{\alpha} L_x^{-1})L_x) (w) )=0 \\
 \iff &\mathit{Z} ( L_x(\sigma_m-\tilde{\alpha}\sigma_m\tilde{\alpha}) (w) )=0. \qedhere
 \end{align*}  
\end{proof}

\section{Derivation relation and Ohno type relation for finite multiple zeta values and their equivalence}
The derivation relation for FMZVs is conjectured by the third author and proved by the second author \cite{murahara_2016}. 
We define two $\Q$-linear maps $\mathit{Z}_{\mathcal{A}} \colon\frak{H}^1 \to \A$
 and $\mathit{Z}_{\mathcal{S}} \colon\frak{H}^1 \to \mathcal{Z}_{\R}/(\zeta (2))$ respectively by $\mathit{Z}_{\mathcal{A}} (1)=1$
 and $\mathit{Z}_{\mathcal{A}} (z_{k_1} \cdots z_{k_r})=\zeta_{\mathcal{A}} (k_1,\ldots ,k_r)$, and $\mathit{Z}_{\mathcal{S}} (1)=1$
 and $\mathit{Z}_{\mathcal{S}} (z_{k_1} \cdots z_{k_r})=\zeta_{\mathcal{S}} (k_1,\ldots ,k_r)$. 
 For notational simplicity, we write $\mathcal{F}=\mathcal{A} \textrm{ or } \mathcal{S}$.
\begin{thm}[Derivation relation for FMZVs; Murahara \cite{murahara_2016}]  \label{der_FMZVs}
 For $l\in\Z_{\ge1}$, we have 
 \begin{align*} \label{1}
  \ZF (L_x^{-1}\partial_{l}L_x(w)) = 0 \quad (w \in \frak{H}^1).  
 \end{align*}
\end{thm} 
\begin{rem}
 Jarossay mentioned the exsistence of the lift of Theorem \ref{der_FMZVs} for $\boldsymbol{p}$-adic $\mathcal{A}$-FMZVs (for details, see Jarossay \cite{jarossay_2017}).
\end{rem}
The third author \cite{oyama_2015} proved the Ohno type relation for FMZVs, which was first conjectured by Kaneko \cite{kaneko_2014}. 
\begin{thm}[Ohno type relation for FMZVs; Oyama \cite{oyama_2015}] \label{ohno-type_FMZVs}
 For $(k_1,\dots,k_r) \in \Z_{\ge1}^r$ and $m \in \Z_{\ge 0}$, we have
 \begin{align*}
  \sum_{\substack{e_1 + \cdots + e_r =m \\ e_i \ge 0 \, (1\leq i \le r)}} 
  \zeta_{\F} (k_1+e_1,\dots,k_r+e_r)
  =\sum_{\substack{e_1' + \cdots + e_{r'}' = m \\ e_i' \ge 0 \, (1\leq i \le r') }}
  \zeta_{\F} ((k'_1+e'_1,\dots,k'_{r'}+e'_{r'})^{\vee}),
 \end{align*}
 where $(k'_1,\dots,k'_{r'})=(k_1,\dots,k_r)^{\vee}$ is Hoffman's dual of $(k_1,\dots,k_r)$.
\end{thm}

Then, we get the following result.
\begin{thm} \label{equivalence2} 
 The derivation relation (Theorem \ref{der_FMZVs}) and the Ohno type relation (Theorem \ref{ohno-type_FMZVs})
 are equivalent. 
\end{thm}
\begin{proof}[Proof of Theorem \ref{equivalence2}] 
 The proof is almost the same as Theorem \ref{equivalence1}. 
 We see that the derivation relation and the Ohno type relation can be stated respectively by 
 $\mathit{Z} ( L_x^{-1}(\sigma_m-\tau \sigma_m \tau) L_x  (w) )=0$ for $w\in x \frak{H}y$ and 
 $\mathit{Z} ( (\sigma_m-\tilde{\alpha}\sigma_m\tilde{\alpha})(w))=0$ for $w\in\frak{H}y$. 
 Then, we can similarly prove the theorem by the following equivalence: 
 \begin{align*}
 &\mathit{Z} ( L_x^{-1}(\sigma_m-\tau \sigma_m \tau) L_x  (w) )=0 \\
 \iff &\mathit{Z} ( (\sigma_m-L_x^{-1}\tau \sigma_m \tau L_x) (w) )=0 \\
 \iff &\mathit{Z} ( (\sigma_m-L_x^{-1}(L_x R \tilde{\alpha} L_x^{-1})\sigma_m (L_x R \tilde{\alpha} L_x^{-1})L_x) (w) )=0 \\
 \iff &\mathit{Z} ( (\sigma_m-\tilde{\alpha}\sigma_m\tilde{\alpha}) (w) )=0. \qedhere
 \end{align*} 
\end{proof}

\section{Two alternative proofs of derivation relation for MZVs}
In this section, we prove Theorem \ref{der_MZVs} in several ways. 
We define an automorphism $\phi$ of $\frak{H}$ by $\phi (x)=z=x+y, \phi (y)=-y$.
We also write 
\begin{align*}
 \frak{H}y \ast \frak{H}y=\{ w\ast w' \mid w,w'\in \frak{H}y \}, \\
 \frak{H}y \,\overline{\ast}\, \frak{H}y=\{ w \,\overline{\ast}\, w' \mid w,w'\in \frak{H}y \}. 
\end{align*}
Here, the harmonic product $\ast$ (resp.\ harmonic-star product $\overline{\ast}$) on $\frak{H}^{1}$ is defined by
\begin{align*}
 & 1\ast w= w\ast 1 = w , \,
 z_{k} w_{1} \ast z_{l} w_{2}= z_{k} (w_{1} \ast z_{l} w_{2}) + z_{l} (z_{k} w_{1} \ast w_{2}) + z_{k+l} (w_{1} \ast w_{2}),  \\
 &(resp.\ 
 1 \,\overline{\ast}\, w= w\,\overline{\ast}\, 1 = w , \,
 z_{k} w_{1} \,\overline{\ast}\, z_{l} w_{2}= z_{k} (w_{1} \,\overline{\ast}\, z_{l} w_{2}) + z_{l} (z_{k} w_{1} \,\overline{\ast}\, w_{2}) - z_{k+l} (w_{1} \,\overline{\ast}\, w_{2}) )
\end{align*}
($k,l \in \mathbb{Z}_{\ge1}$ and $w$, $w_{1}$, $w_{2}$ are words in $\frak{H}^{1}$), 
together with $\Q$-bilinearity. 
The harmonic product $\ast$ (resp.\ the harmonic-star product $\overline{\ast}$) is commutative and associative, therefore $\frak{H}^{1}$ is a $\Q$-commutative algebra with respect to $\ast$ (resp.\ $\overline{\ast}$). 
(See also Hoffman \cite{hoffman-97} and Muneta \cite{muneta_2009}.)

In the proofs, we use the linear part of the Kawashima's relation. 
We define the $\Q$-linear map $\mathit{Z}^\star\colon\frak{H}^0 \to \R$ by $\mathit{Z}^\star (1)=1$ and $\mathit{Z}^\star ( z_{k_1} \cdots z_{k_r})= \zeta^\star (k_1,\ldots, k_r)$. 
\begin{thm}[Kawashima's relation; Kawashima \cite{kawashima_2009}] \label{kawashima_lin}
 We have
 \begin{align*}
  L_x(\phi (\frak{H} y \ast \frak{H} y)) \subset \mathrm{ker} \mathit{Z}, \\
  L_x(\tilde{\alpha} (\frak{H} y \,\overline{\ast}\, \frak{H} y))\subset \mathrm{ker}\mathit{Z}^{\star}.
 \end{align*}
\end{thm}

\subsection{The first proof}
Recall $z=x+y$.
We define an automorphism $S_1$ of $\frak{H}$ by $S_1(x)=x, S_1(y)=z$, and 
$S$ the $\Q$-linear map of $\frak{H} y$ satisfying $S(1)=1$ and $S(wy) = S_1(w)y \,\, (w\in\frak{H})$.
We also set an automorphism $\tilde{S}_1$ of $\frak{H}$ by $\tilde{S}_1(x)=x, \tilde{S}_1(y)=y-x$, and 
$\tilde{S}$ the $\Q$-linear map of $\frak{H} y$ satisfying $\tilde{S}(1)=1$ and $\tilde{S}(wy) = \tilde{S}_1(w)y \,\, (w\in\frak{H})$.
We note that 
\begin{align*}
 & S \circ \tilde{S} = \tilde{S} \circ S = id., \\
 & \mathit{Z}^\star = \mathit{Z} \circ S, \quad \mathit{Z}= \mathit{Z}^\star \circ \tilde{S}.　
\end{align*}
Ihara-Kajikawa-Ohno-Okuda \cite{ihara_kajikawa_ohno_okuda_2011} shows the equivalence of the following Theorem \ref{ikoo} and the derivation relation for MZVs (Theorem \ref{der_MZVs}). 
To prove Theorem \ref{der_MZVs}, we show Theorem \ref{ikoo} instead. 
\begin{thm}[Ihara-Kajikawa-Ohno-Okuda \cite{ihara_kajikawa_ohno_okuda_2011}] \label{ikoo}
For $l \in \Z_{\ge1}$, we have
 \begin{align*}
  \mathit{Z}^{\star} ( \tilde{S} \partial_l S (w) )=0 \quad (w \in \frak{H}^0). 
 \end{align*}
\end{thm}
\begin{proof}
 For $w=xy^{k_1-1} \cdots xy^{k_r-1} y \,\,(k_1,\dots,k_r\ge1)$, we have
 \begin{align*}
  \tilde{S} \partial_l S (w) 
  &= \tilde{S} \partial_l (xz^{k_1-1}\dots xz^{k_r-1}y) \\
  &= \tilde{S} (xz^{l-1}yz^{k_1-1}xz^{k_2-1}\dots xz^{k_r-1}y +\cdots\cdots \\
   &\quad\quad +xz^{k_1-1}\cdots xz^{k_{r-1}-1} xz^{l-1}y z^{k_r-1}y - xz^{k_1-1}\cdots xz^{k_r-1} xz^{l-1}y) \\ 
  &= xy^{l-1}(y-x)y^{k_1-1}xy^{k_2-1}\cdots xy^{k_r} +\cdots\cdots \\
   &\quad +xy^{k_1-1}\cdots xy^{k_{r-1}-1} xy^{l-1}(y-x) y^{k_r} - xy^{k_1-1}\cdots xy^{k_r-1} xy^{l} \\ 
  &= L_x\tilde{\alpha} (x^{l-1}(x-y)x^{k_1-1}y \cdots x^{k_r-1}y +\cdots\cdots \\
   &\quad\qquad\, +x^{k_1-1}y\cdots x^{k_{r-1}-1} yx^{l-1}(x-y) x^{k_r-1}y - x^{k_1-1}y\cdots x^{k_r-1} yx^{l-1}y) \\ 
  &= -L_x\tilde{\alpha} (z_{k_1}\cdots z_{k_r} \,\overline{\ast}\, z_{l}).
 \end{align*}
 By Theorem \ref{kawashima_lin}, we have the desired result.
\end{proof}

\subsection{The second proof}
We prove the following theorem instead of Theorem \ref{der_MZVs}. 
\begin{thm} \label{der_MZVs_ext} 
 For $(m_1,\ldots ,m_n)\in\Z_{\ge1}^{n} \,\, (n\ge 0)$ and $l\in\Z_{\ge1}$, we have
 \begin{align*}
  &\mathit{Z} (xz^{m_1-1}y \cdots z^{m_n-1}y \partial_{l}(w)) \\
  &=-\mathit{Z} (xz^{l-1}y z^{m_{1}-1}y \cdots z^{m_n-1}yw) \\
  &\,\,\,\,\, +\sum_{i=1}^{n} \mathit{Z} (xz^{m_1-1}y \cdots z^{m_i-1}x z^{l-1}y z^{m_{i+1}-1}y \cdots z^{m_n-1}yw) \quad (w \in \frak{H}y). 
 \end{align*}
 When $n=0$, we understand the left-hand side is $\mathit{Z} (x\partial_{l}(w))$ and the right-hand side is $-\mathit{Z} (xz^{l-1}yw)$. 
\end{thm}
\begin{rem}
 We note that Theorem \ref{der_MZVs} is essentially equivalent to Theorem \ref{der_MZVs_ext}. 
 Since $\partial_{l} (xw')=xz^{l-1}yw'+x\partial_{l} (w')$ for $w'\in\frak{H}y$, 
 we see that the statement of Theorem \ref{der_MZVs} can be replaced to
 \begin{eqnarray}  
  \mathit{Z} (x\partial_{l}(w')) = -\mathit{Z} (xz^{l-1}yw') \quad (w'\in\frak{H}y). 
  \label{der_MZVs_ver2}
 \end{eqnarray}
 Thus, Theorem \ref{der_MZVs_ext} (the case $n=0$) implies Theorem \ref{der_MZVs}.
 On the other hand, since 
 \begin{align*} 
  \partial_{l}(z^{m_1-1}y\cdots z^{m_n-1}yw) =&-z^{m_1-1}xz^{l-1}yz^{m_2-1}y\cdots z^{m_n-1}yw \\ 
  &-\quad\cdots\cdots \\
  &-z^{m_1-1}y\cdots z^{m_n-1}xz^{l-1}yw \\ 
  &+z^{m_1-1}y\cdots z^{m_n-1}y\partial_{l}(w)   
 \end{align*}
 (note $\partial_{l}(z)=0$), we get Theorem \ref{der_MZVs_ext} by putting $z^{m_1-1}y\cdots z^{m_n-1}yw$
 for $w'$ in eq.$(\ref{der_MZVs_ver2})$. 
\end{rem}

\begin{proof}[Proof of Theorem \ref{der_MZVs_ext}]
We notice that any word in $\frak{H}y$ can be written by the linear combination of the words $z^{k_1-1}y\cdots z^{k_r-1}y$, e.g., $yxy=yzy-y^3$. 
Thus, we need to prove the theorem only for the word of the form $w=z^{k_1-1}y\cdots z^{k_r-1}y $.

For $w=z^{k_1-1}y\cdots z^{k_r-1}y $, we call $r$ the length of $w$. 
We prove the theorem by induction on the length of $w=z^{k_1-1}y\cdots z^{k_r-1}y$. 

\noindent (I) When $w=z^{k_1-1}y$, we need to show 
\begin{align*}
 -\mathit{Z} ( xz^{m_{1}-1}y \cdots z^{m_n-1}y z^{k_1-1} xz^{l-1}y)
 =&-\mathit{Z} (xz^{l-1}y z^{m_{1}-1}y \cdots z^{m_n-1}y z^{k_1-1}y) \\
 &+\sum_{i=1}^{n} \mathit{Z} (xz^{m_1-1}y \cdots z^{m_i-1}x z^{l-1}y z^{m_{i+1}-1}y \cdots z^{m_n-1}y z^{k_1-1}y).
\end{align*}
We see by Theorem \ref{kawashima_lin}, 
\begin{align*}
 &\mathit{Z} ( xz^{m_{1}-1}y \cdots z^{m_n-1}y z^{k_1-1} xz^{l-1}y) -\mathit{Z} (xz^{l-1}y z^{m_{1}-1}y \cdots z^{m_n-1}y z^{k_1-1}y) \\
 &+\sum_{i=1}^{n} \mathit{Z} (xz^{m_1-1}y \cdots z^{m_i-1}x z^{l-1}y z^{m_{i+1}-1}y  \cdots z^{m_n-1}y z^{k_1-1}y) \\
 &=\mathit{Z} (-xz^{l-1}yz^{m_1-1}y\cdots z^{m_n-1}y z^{k_1-1}y +xz^{m_1-1}xz^{l-1}yz^{m_2-1}y\cdots z^{m_n-1}y z^{k_1-1}y \\
 &\quad\quad\,+\quad\cdots\cdots\quad +xz^{m_1-1}y\cdots z^{m_n-1}xz^{l-1}y z^{k_1-1}y +xz^{m_{1}-1}y \cdots z^{m_n-1}y z^{k_1-1} xz^{l-1}y ) \\
 &=(-1)^{n+1} \mathit{Z} (L_x \phi (x^{l-1}yx^{m_1-1}y\cdots x^{m_n-1}y x^{k_1-1}y
  +x^{m_1-1}zx^{l-1}yx^{m_2-1}y\cdots x^{m_n-1}y x^{k_1-1}y \\
 &\qquad\quad\quad\quad\,\,\,\, +\quad\cdots\cdots\quad
  +x^{m_1-1}y\cdots x^{m_n-1}zx^{l-1}y x^{k_1-1}y +x^{m_{1}-1}y \cdots x^{m_n-1}y x^{k_1-1} zx^{l-1}y)) \\ 
 &=(-1)^{n+1} \mathit{Z} (L_x \phi (x^{m_1-1}y\cdots x^{m_n-1}y x^{k_1-1}y \ast x^{l-1}y))=0. 
\end{align*}
Here, we note that the proof is also valid when $n=0$.
 
\noindent (II) We assume the identity holds for $r-1$. 
 For $w=z^{k_1-1}yw' \,(w'=z^{k_2-1}y\cdots z^{k_r-1}y)$,  
\begin{align*}
 \textrm{L.H.S.}&=\mathit{Z} (xz^{m_1-1}y\cdots z^{m_n-1}y\partial_{l}(z^{k_1-1}yw')) \\
 &=\mathit{Z} (-xz^{m_1-1}y\cdots z^{m_n-1}yz^{k_1-1}xz^{l-1}yw' +xz^{m_1-1}y\cdots z^{m_n-1}yz^{k_1-1}y\partial_{l}(w')). 
\end{align*}
By the induction hypothesis, we have
\begin{align*}
 \mathit{Z} (xz^{m_1-1}y\cdots z^{m_n-1}yz^{k_1-1}y\partial_{l}(w')) &=\mathit{Z} (-xz^{l-1}yz^{m_1-1}y\cdots z^{m_n-1}yz^{k_1-1}yw' \\
 &\quad\quad\, +xz^{m_1-1}xz^{l-1}yz^{m_2-1}y\cdots z^{m_n-1}yz^{k_1-1}yw' \\ 
 &\quad\quad\, +\quad\cdots\cdots \\
 &\quad\quad\, +xz^{m_1-1}y\cdots z^{m_n-1}yz^{k_1-1}xz^{l-1}yw'). 
\end{align*}
Thus, we find
\begin{align*}
 \textrm{L.H.S.}&=\mathit{Z} (-xz^{m_1-1}y\cdots z^{m_n-1}yz^{k_1-1}xz^{l-1}yw' -xz^{l-1}yz^{m_1-1}y\cdots z^{m_n-1}yz^{k_1-1}yw' \\
 &\quad\quad\, +xz^{m_1-1}xz^{l-1}yz^{m_2-1}y\cdots z^{m_n-1}yz^{k_1-1}yw' +\quad\cdots\cdots \\
 &\quad\quad\, +xz^{m_1-1}y\cdots z^{m_n-1}xz^{l-1}yz^{k_1-1}yw' +xz^{m_1-1}y\cdots z^{m_n-1}yz^{k_1-1}xz^{l-1}yw') \\
 &=\mathit{Z} (-xz^{l-1}yz^{m_1-1}y\cdots z^{m_n-1}yz^{k_1-1}yw' +xz^{m_1-1}xz^{l-1}yz^{m_2-1}y\cdots z^{m_n-1}yz^{k_1-1}yw' \\ 
 &\quad\quad\, +\quad\cdots\cdots\quad +xz^{m_1-1}y\cdots z^{m_n-1}xz^{l-1}yz^{k_1-1}yw') \\
 &=\textrm{R.H.S.} \qedhere
\end{align*}
\end{proof}

\section{Three alternative proofs of derivation relation for FMZVs}
In this section, we present three different proofs of Theorem \ref{der_FMZVs}. 
The following theorem gives the important properties of FMZVs.
The first equality for $\A$-FMZVs was proved by Hoffman \cite{hoffman_2015} and the others were obtained by Kaneko-Zagier \cite{kaneko_2014,kaneko_zagier_2017}.  
\begin{thm}[Hoffman \cite{hoffman_2015}, Kaneko-Zagier \cite{kaneko_2014,kaneko_zagier_2017}] \label{dshF}
 For $w=z_{k_1}\cdots z_{k_r}$ and $w'=z_{k'_1}\cdots z_{k'_s}\in\frak{H}^{1}$, we have
 \begin{align*}
  &\ZF (w\ast w')=\ZF(w)\ZF(w'), \\
  &\ZF (w\sh w')=(-1)^{|w|}\ZF(z_{k_r}\cdots z_{k_1}z_{k'_1}\cdots z_{k'_s}),
 \end{align*}
 where $|w|$ denotes the degree of the word $w$ in $x$ and $y$.
\end{thm}

In the proofs, we use the following relations, which are called the duality formulas for FMZ(S)Vs. 
\begin{lem} [Hoffman \cite{hoffman_2015}, Jarossay \cite{jarossay_2014}] \label{dualF}
 For $w \in\frak{H}^{1}$, we have 
 \begin{align*}
  \ZF (w)&=\ZF (\phi (w)), \\
  \ZF^{\star} (w)&=-\ZF^{\star} (\tilde{\alpha} (w)).
 \end{align*} 
 Here, the $\Q$-linear map $\ZF^\star\colon\frak{H}^1 \to \A \,\text{or}\, \mathcal {Z}_{\R}/(\zeta(2))$ is defined by $\ZF^\star (1)=1$ and $\ZF^\star ( z_{k_1} \cdots z_{k_r})= \zeta_\mathcal{F}^\star (k_1,\ldots, k_r)$. 
\end{lem}

\subsection{The first proof}
We define an derivation $\delta_l$ on $\frak{H}$ as a $\Q$-linear map $\delta_l\colon\frak{H}\to\frak{H}$ satisfying Leibniz's rule, and $\delta_l(x)=0$ and $\delta_l(y)=zx^{l-1}y$. 
\begin{lem}[Another type of the derivation relation for FMZVs]
For $l\in\Z_{\ge1}$, we have 
 \begin{align*} 
  \ZF (L_z^{-1}\delta_{l} L_z(w)) = 0 \quad (w \in \frak{H}y).  
 \end{align*}
\end{lem}
\begin{proof}
 Let $w=x^{k_1-1}y\cdots x^{k_r-1}y$. 
 Since
 \begin{align*}
  L_z^{-1}\delta_{l} L_z(w) 
  &=x^{l-1}yx^{k_1-1}y\cdots x^{k_r-1}y+x^{k_1-1}zx^{l-1}yx^{k_2-1}y\cdots x^{k_r-1}y+\cdots \\
  &\quad +x^{k_1-1}y\cdots x^{k_{r-1}-1}yx^{k_r-1}zx^{l-1}y \\
  &=x^{k_1-1}y\cdots x^{k_r-1}y\ast x^{l-1}y \\
  &=w\ast x^{l-1}y,
 \end{align*}
 we find 
 \begin{align*}
  \ZF (L_z^{-1}\delta_{l} L_z(w)) 
  &=\ZF (w\ast x^{l-1}y) \\
  &=\ZF(w) \ZF(x^{l-1}y)=0
 \end{align*}
 by Theorem \ref{dshF} and the fact $\ZF(x^{l-1}y)=0$.
\end{proof}
\begin{proof}[Proof of Theorem \ref{der_FMZVs}]
Since $\partial_l=-\phi \circ \delta_l \circ \phi$, the first statement of Lemma \ref{dualF}, and the above lemma, 
we have
 \begin{align*}
  \ZF (L_x^{-1}\partial_{l}L_x(w))
  &=-\ZF (L_x^{-1}(\phi \circ \delta_l \circ \phi)L_x(w)) \\
  &=-\ZF (\phi L_z^{-1}\delta_{l}L_z\phi(w)) \\
  &=-\ZF (L_z^{-1}\delta_{l}L_z(\phi(w)))=0 
 \end{align*}
 for $w \in \frak{H}y$.
 Thus, we find Theorem \ref{der_FMZVs} holds.
\end{proof}

\subsection{The second proof} 
The similar argument of Ihara-Kajikawa-Ohno-Okuda \cite{ihara_kajikawa_ohno_okuda_2011} shows the equivalence of the following Theorem \ref{derF_star} and the derivation relation for FMZVs.
\begin{thm} \label{derF_star}
For $l \in \Z_{\ge1}$, we have
 \begin{align*}
  \ZF^{\star} ( \tilde{S} \partial_l S (w) )=-\ZF^{\star} (y^{l-1}(y-x)w) \quad (w \in \frak{H}^1). 
 \end{align*}
\end{thm}
Now, we prove this theorem instead of Theorem \ref{der_FMZVs}.
\begin{proof}
 For $w=y^{k_1-1}xy^{k_2-1} \cdots xy^{k_r-1} y \,\,(k_1,\dots,k_r\ge1)$, 
 \begin{align*}
  \tilde{S} \partial_l S (w)  
  &= \tilde{S} \partial_l (z^{k_1-1}xz^{k_2-1} \cdots xz^{k_r-1}y) \\
  &= \tilde{S} (z^{k_1-1}xz^{l-1}yz^{k_2-1}xz^{k_3-1} \cdots xz^{k_r-1}y +\quad\cdots\cdots \\
   &\quad\quad +z^{k_1-1}xz^{k_2-1} \cdots xz^{k_{r-1}-1} xz^{l-1}y z^{k_r-1}y - z^{k_1-1}xz^{k_2-1} \cdots xz^{k_r-1} xz^{l-1}y) \\ 
  &= y^{k_1-1}xy^{l-1}(y-x)y^{k_2-1}xy^{k_3-1}\cdots xy^{k_r} +\quad\cdots\cdots \\
   &\quad +y^{k_1-1}xy^{k_2-1} \cdots xy^{k_{r-1}-1} xy^{l-1}(y-x) y^{k_r} - y^{k_1-1}xy^{k_2-1} \cdots xy^{k_r-1} xy^{l}.  
 \end{align*}
By Theorem \ref{dshF} and Lemma \ref{dualF}, we have
 \begin{align*}
  &\ZF^{\star}( y^{l-1}(y-x)w +\tilde{S} \partial_l S (w) ) \\
  &= \ZF^{\star}( \tilde{\alpha} (x^{l-1}(x-y)x^{k_1-1}y \cdots x^{k_r-1}y +\quad\cdots\cdots \\
   &\qquad\quad\,\,\,\, +x^{k_1-1}y\cdots x^{k_{r-1}-1} yx^{l-1}(x-y) x^{k_r-1}y - x^{k_1-1}y\cdots x^{k_r-1} yx^{l-1}y) ) \\ 
  &= -\ZF^{\star}( \tilde{\alpha} (z_{k_1}\cdots z_{k_r} \,\overline{\ast}\, z_{l}) ) \\
  &= \ZF^{\star}(z_{k_1}\cdots z_{k_r} \,\overline{\ast}\, z_{l}) \\
  &= \ZF^{\star}(z_{k_1}\cdots z_{k_r}) \ZF^{\star}(z_{l}) =0.
 \end{align*}
 Thus, we find the desired result.
\end{proof}

\subsection{The third proof} 
In this subsection, we introduce the proof obtained by Ihara. 
\begin{lem}[Ihara-Kaneko-Zagier {\cite{ihara_kaneko_zagier_2006}}] 
For $w\in\frak{H}^{1}$, we have
 \begin{align*} 
 \frac{1}{1-yu} \ast w=\frac{1}{1-yu} \sh\Delta_u(w),
 \end{align*}
where $\Delta_u$ is the automorphism of $\hat{ \frak{H} }$ given by
 \begin{align*}
 \Delta_u=\exp \left( \sum_{l=1}^\infty \frac{\partial_l}{l}(-u)^l \right),
 \end{align*}
and $u$ is a formal parameter.
\end{lem}

\begin{proof}[Proof of Theorem \ref{der_FMZVs} obtained by Ihara]
 From Ihara-Kaneko-Zagier {\cite{ihara_kaneko_zagier_2006}}, we note that 
 \begin{align*}
  \Delta_u(x)=x(1+yu)^{-1},\quad\Delta_u(y)=y+x(1+yu)^{-1}yu.
 \end{align*}
By theorem \ref{dshF}, we have
 \begin{align*} 
  \ZF \left(\frac{1}{1-yu} \ast w \right)=\ZF \left( \frac{1}{1-yu} \right) \ZF(w)=\ZF(w),  
 \end{align*}
and 
 \begin{align*} 
 \ZF\left(\frac{1}{1-yu} \sh\Delta_u(w)\right) 
 &=\ZF\left(\frac{1}{1+yu} \Delta_u(w)\right) \\
 &=\ZF\left(L_x^{-1}\Delta_u(x)\Delta_u(w)\right) \\
 &=\ZF\left(L_x^{-1}\Delta_u(xw)\right) \\
 &=\ZF\left(L_x^{-1}\Delta_u L_x(w)\right).
 \end{align*}
Then, we have 
 \begin{align*} 
 \ZF\left(L_x^{-1} (1-\Delta_u) L_x(w)\right) =0.
 \end{align*}
This finishes the proof. 
\end{proof}

\section*{Acknowledgements}
The authors would like to thank Professor Kentaro Ihara for sending us the proof of the derivation relation (the third proof for finite multiple zeta values).


\begin{thebibliography}{100}
\bibitem{hoffman-97} M. E. Hoffman, \textit{The algebra of multiple harmonic series}, J.\ Algebra \textbf{194} (1997), 477--495. 
\bibitem{hoffman_2015} M. E. Hoffman, \textit{Quasi-symmetric functions and mod p multiple harmonic sums}, Kyushu J.\ Math.\ \textbf{69} (2015), 345--366. 
\bibitem{ihara_kajikawa_ohno_okuda_2011} K. Ihara, J. Kajikawa, Y. Ohno and J.\ Okuda, \textit{Multiple zeta values vs. multiple zeta-star values}, J.\ Algebra \textbf{322} (2011), 187--208. 
\bibitem{ihara_kaneko_zagier_2006} K. Ihara, M. Kaneko and D. Zagier, \textit{Derivation and double shuffle relations for multiple zeta values}, Compositio Math.\ \textbf{142} (2006), 307--338. 
\bibitem{jarossay_2014} D.\ Jarossay, \textit{Double m\'elange des multiz\^etas finis et multiz\^etas sym\'etris\'es}, C.\ R.\ Acad.\ Sci.\ Paris, \textbf{352} (2014), 767--771. 
\bibitem{jarossay_2017} D.\ Jarossay, \textit{An explicit theory of $\pi_{1}^{\textrm{un,crys}}(\mathbb{P}^{1}-\{0,\mu_{N},\infty\})$ II-1:Standard algebraic equations of prime weighted multiple harmonic sums and adjoint multiple zeta values}, 
arXiv:1412.5099. 
\bibitem{kaneko_2014} M. Kaneko, \textit{Finite multiple zeta values (in Japanese)},  RIMS  K\^oky\^uroku Bessatsu  \textbf{B68} (2017), 175–-190. 
\bibitem{kaneko_zagier_2017} M. Kaneko and D.\ Zagier, \textit{Finite multiple zeta values}, in preparation. 
\bibitem{kawashima_2009} G. Kawashima, \textit{A class of relations among multiple zeta values}, J.\ Number Theory \textbf{129} (2009), 755--788. 
\bibitem{muneta_2009} S.\ Muneta, \textit{Algebraic setup of non-strict multiple zeta values}, Acta Arithmetica \textbf{136} (2009), 7--18. 
\bibitem{murahara_2016} H.\ Murahara, \textit{Derivation relations for finite multiple zeta values}, Int.\ J.\ Number Theory \textbf{13} (2017), 419--427. 
\bibitem{ohno_99} Y. Ohno, \textit{A generalization of the duality and sum formulas on the multiple zeta values}, J.\ Number Theory \textbf{74} (1999), 39--43. 
\bibitem{oyama_2015} K. Oyama, \textit{Ohno-type relation for finite multiple zeta values}, Kyushu J.\ Math.\ \textbf{72} (to appear). 
\bibitem{tanaka_2009} T. Tanaka, \textit{On the quasi-derivation relation for multiple zeta values}, J.\ Number Theory \textbf{129} (2009), 2021--2034. 
\end{thebibliography}
\end{document}